\documentclass[11pt,a4paper]{article}
\usepackage[T1]{fontenc}
\usepackage{lmodern}
\usepackage[margin=1in]{geometry}
\usepackage{amsmath,amssymb,amsthm,mathtools,bm,mathrsfs}
\usepackage{microtype,booktabs,array,needspace}
\usepackage[hidelinks]{hyperref}
\numberwithin{equation}{section}
\newtheorem{theorem}{Theorem}[section]
\newtheorem{proposition}[theorem]{Proposition}
\newtheorem{lemma}[theorem]{Lemma}
\newtheorem{corollary}[theorem]{Corollary}
\newtheorem*{curvaturelemma}{Algebraic curvature lemma}
\theoremstyle{definition}

\theoremstyle{remark}

\newcommand{\dd}{\mathrm d}
\newcommand{\eps}{\epsilon}
\newcommand{\End}{\operatorname{End}}
\newcommand{\tr}{\operatorname{tr}}
\newcommand{\Ric}{\operatorname{Ric}}
\newcommand{\Hol}{\operatorname{Hol}}
\newcommand{\id}{\operatorname{id}}
\newcommand{\im}{\operatorname{im}}
\newcommand{\diag}{\operatorname{diag}}
\newcommand{\CS}{\operatorname{CS}}
\newcommand{\cC}{\mathcal C}
\newcommand{\cK}{\mathcal K}
\newcommand{\cB}{\mathscr B}
\newcommand{\cR}{\mathscr R}
\newcommand{\ff}{\mathfrak f}
\newcommand{\bb}{\mathfrak b}
\newcommand{\cc}{\mathfrak c}
\newcommand{\dE}{\dd_E}
\newcommand{\dn}{\dd_\nabla}
\newcommand{\R}{\mathbb R}

\DeclareMathOperator{\vol}{vol}
\allowdisplaybreaks[1]
\hypersetup{pdftitle={Chain-level decoupling of the physical heterotic G2 deformation complex at the standard embedding},pdfauthor={Bram Brongers}}
\title{Chain-level decoupling of the physical heterotic $G_2$ deformation complex at the standard embedding}
\author{
Bram Brongers\\
\href{mailto:brambrongers98@gmail.com}{\texttt{brambrongers98@gmail.com}}
}
\date{}
\begin{document}
\maketitle
\begin{abstract}
The physical deformation complex of heterotic $G_2$ compactifications
couples instanton deformations to geometric fields and their first-order
$\alpha'$ corrections. We give an explicit chain-level decoupling of
its first-order coefficient complex at the minimal standard embedding, where the gauge bundle is the tangent bundle, the gauge connection is
the Levi--Civita connection, and the background $G_2$ structure is torsion-free.
The first Atiyah coupling is null-homotopic through a skew covariant
derivative extending the induced connection variation. A second
homotopy, given by the covariant codifferential of the skew part of an
endomorphism-valued cochain, removes the coupling from the bundle sector to the geometric sector. The residual geometric curvature term cancels
against the contribution of the first transformation. These identities
hold on arbitrary cochains in every canonical degree and yield mutually
inverse differential-operator chain maps. Consequently, the cohomology
decomposes into instanton and geometric coefficient cohomologies in all
degrees. The argument requires only $\Hol(g)\subseteq G_2$. To our knowledge, no previous explicit all-degree splitting removes the internal gauge and geometric couplings of this physically reduced standard-embedding coefficient complex.
\end{abstract}

\section{Introduction}
\label{sec:introduction}

The standard embedding \cite{CHSW} relates two connections which enter heterotic
compactification in different ways. The gauge connection controls
instanton deformations, whereas the tangent connection enters both
the geometry and the anomaly equation. At the standard embedding one
identifies the gauge bundle with $TY$ and takes the gauge connection
to be the tangent connection entering the anomaly term. At the
torsion-free background considered here this gives
$A=\nabla^{LC}$ and $F=R$, so the gauge and gravitational
characteristic forms agree when the invariant bilinear forms are
normalized consistently. The corresponding linearized deformation equations nevertheless remain coupled. The linearized
instanton equation couples a change of connection to the change of
the geometric structure defining the instanton projection. At first
order in $\alpha'$, the geometric equations also contain the gauge
curvature pairing and a differential curvature correction. The question remains whether these couplings are nontrivial at the level of the deformation complex.

For a torsion-free $G_2$ manifold equipped with an instanton bundle,
de la Ossa, Larfors and Svanes (DLS) constructed a $G_2$ analogue \cite{DLS} of the extension introduced by Atiyah \cite{Atiyah}. Its off-diagonal curvature map measures
whether a geometric deformation preserves the possibility of an
instanton connection. The relevant cochains belong to the canonical
$G_2$ complex, with cochains in $\Lambda^0$, $\Lambda^1$, $\Lambda^2_7$ and $\Lambda^3_1$: these are functions, one-forms, two-forms in the seven-dimensional representation and three-forms in the trivial representation, respectively
\cite{Carrion,FernandezUgarte}. This supplies a natural language for
deformations of geometry and gauge fields on a seven-manifold.
The extension to heterotic $G_2$ systems incorporates the anomaly
equation and a tangent-connection sector \cite{DLSHeterotic}.

An independent variation of the tangent connection enlarges the
field space beyond the physical metric, two-form and gauge
fluctuations. McOrist, Sticka and Svanes (MSS) eliminate that independent
variable and obtain a physical differential with a curvature times
covariant derivative term \cite{MSS}. Their order-by-order
reformulation is a double extension: the instanton complex is
coupled to the leading geometric complex, and this pair is in turn
coupled to the first-order geometric coefficient complex. Their construction gives the physical double extension and sketches the associated cohomology problem, which is left largely open \cite[Section~7]{MSS}. At the standard embedding, the theorem below explicitly splits these couplings in every degree.  

In the Calabi--Yau setting, Chisamanga, McOrist, Picard and Svanes
proved a standard-embedding cohomology decomposition for the
physical heterotic extension \cite{CMPSS}. This suggests a
corresponding question for $G_2$ geometry, but the complex-geometric
argument cannot be transferred by replacing $\bar\partial$ with
the canonical $G_2$ differential. The latter contains nontrivial intermediate projections, while the
curvature of the induced connection acts on both the form and
coefficient factors. Thus an all-degree chain identity must account
for both effects.

We give an explicit chain-level decoupling at the minimal standard
embedding
\begin{equation}\label{eq:intro-standard}
H=0,\qquad V\simeq TY,\qquad A=\nabla^{LC},\qquad F=R,
\end{equation}
on a torsion-free $G_2$ manifold. More precisely, let $G_0^\bullet$ and $G_1^\bullet$ denote the two
geometric coefficient complexes, and let
$E^\bullet=\cC^\bullet(\End_0(TY))$ be the instanton complex.
The coefficient differential has the form
\begin{equation}\label{eq:intro-D}
D^k=
\begin{pmatrix}
d_1^k&b^k&c^k\\
0&e^k&f^k\\
0&0&d_0^k
\end{pmatrix}.
\end{equation}
Theorem~\ref{thm:main} constructs mutually inverse
differential-operator chain maps giving
\begin{equation}\label{eq:intro-theorem}
(B_\tau^\bullet,D)\cong
(G_1^\bullet,d_1)\oplus(E^\bullet,e)\oplus(G_0^\bullet,d_0).
\end{equation}
The isomorphism in every canonical degree is given by differential operators on smooth sections rather than by a $C^\infty(Y)$-linear bundle map. Only $\Hol(g)\subseteq G_2$ is needed; compactness
and full $G_2$ holonomy play no role in the construction.

Three operator identities lead to this result. For a \(T^*Y\)-valued cochain \(m\), the skew covariant derivative

\[
(\cK m)(X,Y)
=
(\nabla_Ym)(X)
-
(\nabla_Xm)(Y)
\]

extends the induced Levi--Civita connection variation and trivializes the Atiyah coupling:

\[
f=\cK d_0-e\cK.
\]

After shifting the instanton variable by \(\cK m\), the coupling \(b\) is trivialized by the covariant codifferential. Writing \(a^-\) for the metric-skew part of an endomorphism-valued cochain \(a\),

\[
a^-(X,Y)
=
\frac12\bigl(g(X,aY)-g(Y,aX)\bigr),
\]

the homotopy is

\[
L_Ea
=
\frac{\alpha'}2\,
d_\nabla^*a^-,
\qquad
b=d_1L_E-L_Ee,
\]
Here \(a^-\) is regarded canonically as a \(\Lambda^pT^*Y\)-valued two-form.

Finally, the two curvature times derivative terms satisfy the
stronger, unprojected cancellation $c+b\cK=0$. In the two
homotopy calculations, symmetric second derivatives cancel
before curvature identities are used. The remaining curvature
terms from the form factor vanish under the canonical $G_2$
projections. The second calculation also uses Ricci-flatness. The
homotopy identities therefore hold on arbitrary cochains.

The expansion $M=m+\eps x$, with $\eps=\alpha'$ and $\eps^2=0$, determines the geometric coefficient spaces used in \eqref{eq:intro-theorem}. Section~\ref{subsec:weights} specifies how the gauge sector enters the first-order coefficient module. The chain maps preserve its dual-number action.

Related formulations of heterotic $G_2$ deformation theory include the following.
General BPS complexes organize instanton and $G$-structure
deformations, including heterotic systems \cite{BPS};
coupled $G_2$-instantons give a complementary geometric treatment
of the standard embedding \cite{Coupled}. The diagonalization of de la Ossa, Larfors, Magill and Svanes \cite[Section~3.5]{QuantumG2} uses an enlarged superpotential complex and is established at least in homological degree zero. It separates the outer rows while retaining the internally coupled $Q$-valued middle differential. Our theorem instead concerns the physically reduced MSS coefficient complex and removes its internal gauge and geometric couplings by an all-degree chain isomorphism.
Likewise, the decoupling used to establish ellipticity in
\cite[Section~3]{CGT} concerns the highest-order operator.
Here the curvature terms in the complete coefficient
differential are retained and explicitly eliminated.

Section~\ref{sec:geometry} fixes the curvature and projection
conventions. Section~\ref{sec:heterotic} specifies the physical coefficient complex. Sections~\ref{sec:first}
and~\ref{sec:second} prove the two homotopies, including
every canonical degree. Section~\ref{sec:residual} establishes
the residual cancellation and its transgression interpretation.
The full chain isomorphism and its cohomology consequence are
proved in Section~\ref{sec:theorem}. Section~8 compares the Calabi--Yau argument, Section~9 records the physical scope of the result, and Section~10 discusses the nonlinear deformation problem suggested by the linear splitting. The appendices record the projector normalizations, form-degree conversion and coefficient-module details.

\section{Torsion-free \texorpdfstring{$G_2$}{G2} geometry and canonical complexes}
\label{sec:geometry}

\subsection{Forms, curvature and conventions}

Let $Y$ be a smooth seven-manifold with a positive three-form $\varphi$.
Its metric and orientation are denoted by $g$ and $\vol_g$, and
$\psi=*\varphi$. We assume
\begin{equation}\label{eq:torsionfree}
\dd\varphi=0,\qquad \dd\psi=0.
\end{equation}
Equivalently, the Levi--Civita
connection $\nabla$ preserves $\varphi$ and $\psi$ and
$\Hol(g)\subseteq G_2$ \cite{FernandezGray,DLS}.
No compactness or full-holonomy assumption is made.

The metric scalar product on forms is fixed by
$\alpha\wedge*\beta=\langle\alpha,\beta\rangle\vol_g$.
Our normalization is $\varphi\wedge\psi=7\vol_g$; its standard
form and numerical contractions are recorded in
Appendix~\ref{app:realization}.
We identify a two-form $A$ with the skew endomorphism $A^\sharp$
by $g(X,A^\sharp Y)=A(X,Y)$.

We use the curvature convention
\begin{equation}\label{eq:curvature-conventions}
R(X,Y)=[\nabla_X,\nabla_Y]-\nabla_{[X,Y]},
\end{equation}
together with the standard Riemannian curvature symmetries and
Bianchi identities.

The consequences of \eqref{eq:torsionfree} used below are
\begin{equation}\label{eq:background-curvature}
\nabla\varphi=\nabla\psi=0,\qquad \Ric(g)=0,\qquad
R\in\Omega^2_{14}(Y,\mathfrak g_2(TY)).
\end{equation}
The Ricci-flatness is the usual consequence of the parallel spinor
of a torsion-free $G_2$ structure \cite{DLS}.
The identity $\nabla\varphi=0$ makes $R(X,Y)$ valued in
$\mathfrak g_2$. Regard curvature as the self-adjoint operator
$\mathcal R:\Lambda^2TY\to\Lambda^2TY$, using the metric
identification above. Its image lies in $\Lambda^2_{14}$.
By self-adjointness it annihilates the orthogonal summand
$\Lambda^2_7$. Thus every scalar two-form obtained by pairing the
endomorphism value of $R$ with a fixed endomorphism has exterior
type 14. This is the curvature property used by the canonical
projections below. It proves the Levi--Civita instanton condition
without assuming full holonomy.

\subsection{Projectors and the two vanishing identities}

The decompositions are
\begin{equation}\label{eq:decompositions}
\Lambda^2=\Lambda^2_7\oplus\Lambda^2_{14},\qquad
\Lambda^3=\Lambda^3_1\oplus\Lambda^3_7\oplus\Lambda^3_{27}.
\end{equation}
Here
\[
\Lambda^2_7=\{\iota_X\varphi:X\in TY\},\qquad
\Lambda^2_{14}=\{\eta:\eta\wedge\psi=0\}\simeq\mathfrak g_2,
\qquad \Lambda^3_1=\R\varphi.
\]
We use the orthogonal $G_2$-equivariant projections
$\pi^2_7,\pi^2_{14},\pi^3_1,\pi^3_7,\pi^3_{27}$ onto these summands.
Their numerical coefficients are recorded in Appendix~\ref{app:realization}.
Superscripts will be omitted when the exterior degree determines
the projector.

\begin{lemma}\label{lem:projection}
For $\eta\in\Lambda^2_{14}$ and $\lambda\in\Lambda^1$,
\begin{equation}\label{eq:projection-vanishing}
\pi^2_7\eta=0,\qquad
\eta\wedge\psi=0,\qquad
\pi^3_1(\eta\wedge\lambda)=0.
\end{equation}
The same statements hold coefficientwise for bundle-valued forms.
\end{lemma}
\begin{proof}
The first two assertions follow from orthogonality and the
description of $\Lambda^2_{14}$. For the last assertion, the
adjointness of wedge and interior multiplication gives
\begin{equation}\label{eq:scalar-projection-calculation}
\langle\eta\wedge\lambda,\varphi\rangle
=\langle\eta,\iota_{\lambda^\sharp}\varphi\rangle=0,
\end{equation}
since $\iota_{\lambda^\sharp}\varphi\in\Lambda^2_7$ and
$\Lambda^2_7\perp\Lambda^2_{14}$.
Orthogonal projection onto $\R\varphi$ gives the claim. This also proves
$\pi^3_1(\lambda\wedge\eta)=0$, since a one-form and a two-form commute
under the wedge product.
\end{proof}

\subsection{Canonical cochains and bundle differentials}

For a bundle $W$ with connection $\nabla^{W}$, define
\begin{equation}\label{eq:canonical-spaces}
\begin{gathered}
\cC^k(W)=\Gamma(\Lambda^k_{\rm can}T^*Y\otimes W),\\
\Lambda^0_{\rm can}=\Lambda^0,\quad
\Lambda^1_{\rm can}=\Lambda^1,\quad
\Lambda^2_{\rm can}=\Lambda^2_7,\quad
\Lambda^3_{\rm can}=\Lambda^3_1.
\end{gathered}
\end{equation}
All other degrees are zero. Let $\Pi_k$ be the corresponding
projection; $\Pi_0$ and $\Pi_1$ are identities and $\Pi_4=0$.
The canonical differential is
\begin{equation}\label{eq:canonical-differential}
\check\dd_{W}^k=\Pi_{k+1}\dd_{W}\big|_{\cC^k(W)}.
\end{equation}
The scalar complex and its instanton-valued extension originate in
the work of Reyes Carri\'on and Fern\'andez--Ugarte
\cite{Carrion,FernandezUgarte}; the notation follows the
$G_2$ deformation setting of \cite{DLS}.

Suppose the curvature $F_{W}$ has exterior type 14. Then
\begin{equation}\label{eq:canonical-sequence}
0\longrightarrow\cC^0(W)\xrightarrow{\check\dd_{W}}
\cC^1(W)\xrightarrow{\check\dd_{W}}
\cC^2(W)\xrightarrow{\check\dd_{W}}\cC^3(W)
\longrightarrow0
\end{equation}
is a complex. Indeed, on a zero-form $u$ its square is
$\pi_7(F_{W}u)=0$. For a one-form $v$, put
$\eta=\pi_{14}\dd_{W}v$. Since $\nabla\varphi=0$,
$\nabla\pi_{14}=0$; together with Lemma~\ref{lem:projection} this gives
$\pi_1\dd_{W}\eta=\pi_1(e^i\wedge\nabla_i^{W}\eta)=0$. Thus
\begin{equation}\label{eq:canonical-square-one}
\check\dd_{W}^2\check\dd_{W}^1v
=\pi_1\dd_{W}\pi_7\dd_{W}v
=\pi_1(F_{W}\wedge v)-\pi_1\dd_{W}\eta=0.
\end{equation}
There are no further nontrivial two-step compositions.

For $W=T^*Y$ we write the unprojected differential as $\dn$.
For $E=\End_0(TY)$ we write $\dE$ for the exterior covariant
derivative induced by $\nabla$. Both bundles have instanton curvature by
\eqref{eq:background-curvature}.

\section{The heterotic coefficient complex}
\label{sec:heterotic}

\subsection{Standard embedding and invariant form}

Throughout, by the standard embedding we mean a fixed identification of the gauge bundle with $TY$, under which
\begin{equation}\label{eq:standard-embedding}
V=TY,\qquad A=\nabla,\qquad F=R,\qquad H=0.
\end{equation}
We work with the endomorphism coefficient bundle
\begin{equation}\label{eq:E-definition}
E=\End_0(TY)=\ker\bigl(\tr:\End(TY)\longrightarrow\R\bigr).
\end{equation}
 Although $\cK m$ is skew and the background curvature is $\mathfrak g_2$-valued, the coefficient bundle of the complex is the full trace-free endomorphism bundle. The connection on $E$ is induced by $\nabla$.
The restriction from $\End(TY)$ to \eqref{eq:E-definition} is
preserved by the differential: the curvature insertion is skew,
and hence trace-free, and the induced covariant derivative
preserves trace. 

The invariant bilinear form used to write the curvature coupling
is the ordinary matrix trace in the real seven-dimensional
representation:
\begin{equation}\label{eq:tau}
\tau(S,T)=\tr_7(ST).
\end{equation}
For matrix-valued forms, $\tau(\alpha\wedge\beta)$
means $\tr_7(\alpha\wedge\beta)$ with matrix multiplication
in the displayed order. More generally, $\tau$ applied to a matrix-valued product denotes the trace of the corresponding matrix product.
For exterior degrees $p,q$,
$\tau(\alpha\wedge\beta)=(-1)^{pq}\tau(\beta\wedge\alpha)$.
 The same invariant form is used on the gauge and tangent copies in \eqref{eq:standard-embedding}, so at the standard embedding $\tau(F\wedge F)=\tau(R\wedge R)$. 

\subsection{Physical variables and explicit operators}

The physical variable of \cite[Section~4]{MSS} is a
$T^*Y$-valued one-form $M$, combining the metric and
anomaly-trivialized two-form fluctuations, together with
$a=\delta A$. At leading order one may write
$M=\tfrac12(h+\beta)$, with $h\in\Gamma(S^2T^*Y)$ and
$\beta\in\Omega^2(Y)$. Its extension to canonical cochains gives the
complex considered here. The independent tangent-connection
variation has already been eliminated. A
curvature times covariant derivative operator remains in
the geometric diagonal block. This physical elimination
and the same field variables are also described in
\cite[Section~2 and Appendix~A]{MSSMetric}.

Define the endomorphism-valued one-form
\begin{equation}\label{eq:curvature-insertion}
\cR_Z=\iota_ZR,\qquad \cR_Z(X)=R(Z,X).
\end{equation}
For $m\in\Omega^p(Y,T^*Y)$, put $m_X=m(\,\cdot\,,X)$ and
define
\[
\cB:\Omega^p(Y,T^*Y)\longrightarrow\Omega^p(Y,\End(TY)),
\qquad
g\bigl(X,\cB(m)Y\bigr)=(\nabla_Xm)_Y.
\]
For a $T^*Y$-valued $p$-form $m$ and an $E$-valued $p$-form
$a$, define the unprojected operators by
\begin{align}
\ff m=\mathcal F_R(m)
&=\operatorname{tr}_g\bigl((\iota_{\bullet}R)\wedge m_{\bullet}\bigr),
\label{eq:f-unprojected}\\
(\bb a)(Z)&=-\frac{\eps}{4}\tau(\cR_Z\wedge a),
\label{eq:b-unprojected}\\
(\cc m)(Z)&=-\frac{\eps}{2}\tau(\cR_Z\wedge \cB(m)).
\label{eq:c-unprojected}
\end{align}
Here $\operatorname{tr}_g$ denotes contraction of the displayed
$TY$ and $T^*Y$ factors with the metric. We write
$\bb=\eps\bb_0$ and $\cc=\eps\cc_0$, where
$\bb_0,\cc_0$ are the coefficients of $\eps$.
These operators give the invariant-form convention for the
physically reduced differential of \cite[(4.10)--(4.11), Section~7]{MSS}
used throughout.\footnote{
    With the curvature convention \eqref{eq:curvature-conventions} and invariant
    form \eqref{eq:tau}, the reduced gravitational curvature term obtained directly
    from the parent equation and connection variation
    \cite[(4.1), (4.5)--(4.8)]{MSS} has the opposite sign from the corresponding
    term as displayed in \cite[(4.10)--(4.11)]{MSS}, after translation to these
    conventions. We use the sign obtained from the parent equations throughout;
    it is also the sign for which the reduced coefficient differential squares to
    zero.
}

Before coefficient extraction, the two-block physical operator is
\begin{equation}\label{eq:physical-two-block}
\check D=
\Pi\begin{pmatrix}
\dn+\cc&\bb\\
\ff&\dE
\end{pmatrix}.
\end{equation}

\subsection{The weighted coefficient module}
\label{subsec:weights}

Put
\begin{equation}\label{eq:coefficient-ring}
S=\R[\eps]/(\eps^2),\qquad
G^\bullet=\cC^\bullet(T^*Y),\qquad
E^\bullet=\cC^\bullet(E),
\end{equation}
and write $d=\check\dd_\nabla$ and $e=\check\dd_E$.
The order convention of \cite[Section~7]{MSS} retains $M=m+\eps x$ and only the leading gauge coefficient. We encode this convention by letting multiplication by $\eps$ annihilate the gauge sector. The module for the first-order truncated equations is
\begin{equation}\label{eq:weighted-module}
B_\tau^\bullet=(G^\bullet\otimes_\R S)\oplus E^\bullet,
\qquad \eps E^\bullet=0.
\end{equation}
As real spaces, set $G_0^\bullet=G^\bullet$ and
$G_1^\bullet=\eps G^\bullet$ and use the ordered coordinates
$(\eps x,a,m)$. The diagonal differentials are
$d_0m=dm$ and $d_1(\eps x)=\eps dx$.

The projected couplings, with their domains, are
\begin{align}
f^k&=\Pi_{k+1}\ff:
G_0^k\longrightarrow E^{k+1},
\label{eq:f-projected}\\
b^k&=\Pi_{k+1}\bb:
E^k\longrightarrow G_1^{k+1},
\qquad
c^k=\Pi_{k+1}\cc:
G_0^k\longrightarrow G_1^{k+1}.
\label{eq:bc-projected}
\end{align}

All maps out of degree $3$ are zero. Define
\begin{equation}\label{eq:D-definition}
D^k=
\begin{pmatrix}
d_1^k&b^k&c^k\\
0&e^k&f^k\\
0&0&d_0^k
\end{pmatrix}
\quad\text{on }G_1^k\oplus E^k\oplus G_0^k.
\end{equation}
Indeed, expanding the geometric row of
\eqref{eq:physical-two-block} gives
\begin{equation}\label{eq:expanded-equations}
\dn m+\eps\bigl(\dn x+\cc_0m+\bb_0a\bigr),
\end{equation}
and the gauge component is $\dE a+\ff m$.
Applying the canonical projections gives
\eqref{eq:D-definition}.

In unscaled real coordinates $(x,a,m)$, the matrix of
$D$ has first row $(d,b_0,c_0)$, where $b_0,c_0$ are
the coefficients of $\eps$ in $b,c$. Multiplication by
$\eps$ is the fixed map
\begin{equation}\label{eq:N-action}
N(x,a,m)=(m,0,0).
\end{equation}
The equality $DN=ND$ follows because $d_0$ and $d_1$
are the same geometric differential on their respective
coefficients. It makes \eqref{eq:D-definition} an
$S$-linear operator on \eqref{eq:weighted-module}.

Thus the geometric sector is free over $S$, whereas the gauge sector is an $S/(\eps)$-module.

\subsection{The two extensions}
\label{subsec:extensions}

Let $A^\bullet=E^\bullet\oplus G_0^\bullet$, equipped with
\begin{equation}\label{eq:first-extension-D}
D_A^k=\begin{pmatrix}e^k&f^k\\0&d_0^k\end{pmatrix}.
\end{equation}
The first extension is
\begin{equation}\label{eq:first-SES}
0\longrightarrow E^\bullet
\xrightarrow{\,a\mapsto(a,0)\,}A^\bullet
\xrightarrow{\,(a,m)\mapsto m\,}G_0^\bullet
\longrightarrow0.
\end{equation}
The second extension is
\begin{equation}\label{eq:second-SES}
0\longrightarrow G_1^\bullet
\xrightarrow{\,u\mapsto(u,0,0)\,}B_\tau^\bullet
\xrightarrow{\,(u,a,m)\mapsto(a,m)\,}A^\bullet
\longrightarrow0.
\end{equation}
Both sequences are exact. The triangular matrices show
that the inclusions and quotients intertwine the
differentials.

Besides the diagonal squares proved in
Section~\ref{sec:geometry}, the required square identities
are
\begin{equation}\label{eq:three-block-identities}
ef+fd_0=0,\qquad d_1b+be=0,\qquad
d_1c+cd_0+bf=0.
\end{equation}
Adjacent canonical degree labels are implicit in this
display. These are the standard-embedding
double-extension identities. The proofs below establish
stronger formulas from which all three follow
algebraically; see Proposition~\ref{prop:nilpotency}.

\section{Splitting the \texorpdfstring{$G_2$}{G2} Atiyah extension}
\label{sec:first}

\subsection{The skew covariant derivative}

For $m\in\Omega^p(Y,T^*Y)$, define
\begin{equation}\label{eq:K-definition}
\begin{gathered}
\cK:\Omega^p(Y,T^*Y)\longrightarrow
\Omega^p(Y,\mathfrak{so}(TY))\subset\Omega^p(Y,E),\\
(\cK m)(X,Y)=(\nabla_Ym)(X)-(\nabla_Xm)(Y).
\end{gathered}
\end{equation}
Here $X,Y$ are arguments of the $T^*Y$ factor, and the metric
identifies the resulting two-form with a skew endomorphism.
Because $\nabla\Pi_k=0$, $\cK$ commutes with the canonical
projections:
\begin{equation}\label{eq:K-projection}
\Pi_k\cK=\cK\Pi_k.
\end{equation}
Thus its restriction is $\cK^k:G_0^k\to E^k$.

Writing $m_X=m(\,\cdot\,,X)$, the curvature insertion
\eqref{eq:f-unprojected} has the equivalent local expression
\[
(\mathcal F_R(m))(X,Y)=\sum_j e^j\wedge m_{R(X,Y)e_j}.
\]
Indeed, pair symmetry gives
$g(X,R(e_j,V)Y)=g(e_j,R(X,Y)V)$; contraction with $m_{e_j}$
proves the equality. The first Bianchi identity also writes its
coefficient as $R(X,Y)V=R(V,Y)X-R(V,X)Y$, the form used below.
The projected operator is $f^k=\Pi_{k+1}\mathcal F_R$.

For $S\in\Omega^2(Y,TY)$ define the insertion derivation
$\mathcal I_S:\Omega^p(Y)\to\Omega^{p+1}(Y)$ by
\[
\mathcal I_{\beta\otimes X}\alpha=\beta\wedge\iota_X\alpha,
\qquad R_X(U,V)=R(U,V)X.
\]
This definition extends linearly and uses only tensor contraction.

\begin{proposition}\label{prop:K}
On every canonical cochain, the curvature insertion
satisfies
\begin{equation}\label{eq:K-homotopy}
f^k=\cK^{k+1}d_0^k-e^k\cK^k,\qquad k=0,1,2,3.
\end{equation}
Here the outgoing maps at degree $3$ and $\cK^4$ are zero.
\end{proposition}

\begin{proof}
We first calculate before projection in an arbitrary
exterior degree $p$. Choose an orthonormal frame that is normal at the point of calculation, and extend $X,Y$ with vanishing covariant
derivatives there. Expanding the two compositions gives
\begin{align}
((\dE\cK-\cK\dn)m)(X,Y)
&=\sum_i e^i\wedge\bigl(
 ([\nabla_{e_i},\nabla_Y]m)_X
 -([\nabla_{e_i},\nabla_X]m)_Y\bigr).
\label{eq:K-second-jets}
\end{align}
The symmetric second-derivative terms cancel in
\eqref{eq:K-second-jets}; the remaining expression is tensorial.

The curvature of the induced connection on
$\Lambda^pT^*Y\otimes T^*Y$ is the sum of its actions on the form
and $T^*Y$ factors. Let $\rho_p(A)$ be the induced action of an
endomorphism on $p$-forms, so that
$\rho_p(A)\alpha=-\sum_j e^j\wedge\iota_{Ae_j}\alpha$.
Then
\begin{equation}\label{eq:curvature-pform-covector}
([\nabla_U,\nabla_V]m)_X
=-m_{R(U,V)X}+\rho_p(R(U,V))m_X.
\end{equation}
The commutators here are evaluated in the same normal extensions.

The contribution of the first term of
\eqref{eq:curvature-pform-covector} to
\eqref{eq:K-second-jets} is
\[
\sum_i e^i\wedge\bigl(-m_{R(e_i,Y)X}+m_{R(e_i,X)Y}\bigr).
\]
The algebraic Bianchi identity reads
\begin{equation}\label{eq:K-free-Bianchi}
R(e_i,X)Y-R(e_i,Y)X=-R(X,Y)e_i.
\end{equation}
Thus this contribution is $-\mathcal F_R(m)$, with exactly the
sign in \eqref{eq:f-unprojected}.

For the form-factor contribution, a second use of
the same identity gives
\begin{equation}\label{eq:form-Bianchi}
R(U,Y)V-R(V,Y)U=R(U,V)Y.
\end{equation}
Antisymmetrizing the two exterior factors therefore yields
\begin{equation}\label{eq:form-curvature-reduction}
\sum_i e^i\wedge\rho_p(R(e_i,Y))\alpha=-\mathcal I_{R_Y}\alpha.
\end{equation}
Combining the two contributions proves the complete
unprojected identity
\begin{equation}\label{eq:K-unprojected}
\begin{gathered}
(\dE\cK-\cK\dn)m=-\mathcal F_R(m)+\mathscr Q_{\cK}(m),\\
\mathscr Q_{\cK}(m)(X,Y)
=-\mathcal I_{R_Y}m_X+\mathcal I_{R_X}m_Y.
\end{gathered}
\end{equation}

Applying the canonical differential and using \eqref{eq:K-projection},
\begin{equation}\label{eq:K-projected-composition}
e^k\cK^km-\cK^{k+1}d_0^km
=\Pi_{k+1}(\dE\cK-\cK\dn)m
\end{equation}
for $m\in G_0^k$. This equality accounts for the
intermediate projection in the second composition.
We check $\mathscr Q_{\cK}(m)$ in each degree.

For $k=0$, insertion into a zero-form is zero, so it vanishes before
projection. For $k=1$, the inserted form is a scalar;
each summand is a scalar curvature two-form in
$\Lambda^2_{14}$, and its $\Pi_2=\pi_7$ projection
vanishes. For $k=2$, the inserted form is a one-form.
Each summand has the form $\eta_{14}\wedge\lambda_1$,
whose $\Pi_3=\pi_1$ projection is zero by
Lemma~\ref{lem:projection}.
For $k=3$, all outgoing arrows are zero.
These four statements prove \eqref{eq:K-homotopy}.
\end{proof}

\subsection{Chain splitting and exact representatives}\label{subsec:K-splitting}

Let $S_K:A^\bullet\to E^\bullet\oplus G_0^\bullet$ and
$T_K:E^\bullet\oplus G_0^\bullet\to A^\bullet$ be
\begin{equation}\label{eq:first-ST}
S_K^k(a,m)=(a-\cK^km,m),\qquad
T_K^k(a',m)=(a'+\cK^km,m).
\end{equation}
The two compositions cancel the added $\cK^km$ and
are the identity. Applying the differential gives
\begin{align}
D_A^kT_K^k(a',m)
&=\bigl(e^ka'+(e^k\cK^k+f^k)m,d_0^km\bigr)
\notag\\
&=\bigl(e^ka'+\cK^{k+1}d_0^km,d_0^km\bigr)
\notag\\
&=T_K^{k+1}(e^ka',d_0^km).
\label{eq:first-chain-calculation}
\end{align}
Thus these are inverse chain isomorphisms.
For a closed geometric cochain, the primitive for its
Atiyah image is
\begin{equation}\label{eq:first-primitive}
d_0^km=0\quad\Longrightarrow\quad
f^km=-e^k(\cK^km).
\end{equation}
Equivalently, $(\cK^km,m)$ is a closed lift in
$A^k$. This proves vanishing of the first connecting
homomorphism in every degree, on arbitrary closed
representatives.

For a geometric gauge parameter $\xi\in G_0^0$, \eqref{eq:K-homotopy} gives
\begin{equation}\label{eq:K-exact}
\cK^1d_0^0\xi=e^0\cK^0\xi+f^0\xi.
\end{equation}
Thus the induced connection variation decomposes into the Atiyah term $f^0\xi$ and the exact bundle term $e^0\cK^0\xi$. Replacing $m$ by $m+d_0\xi$
changes \eqref{eq:first-primitive} according to
$fd_0=-ef$, as required by the gauge complex.

\subsection{Levi--Civita variation and the instanton equation}
\label{subsec:LC}

For a family of torsion-free structures, we relate $\cK$ to the induced Levi--Civita variation, including the frame and antisymmetric-field contributions.

Let $g_t$ be a metric family and put $h=\dot g$.
For the variation of its connection, write
$L_XY=\left.\frac{d}{dt}\right|_0\nabla^{g_t}_XY$.
Differentiating metric compatibility and zero torsion gives
\begin{equation}\label{eq:LC-linear-system}
\begin{gathered}
(\nabla_Xh)(Y,Z)=g(L_XY,Z)+g(Y,L_XZ),\\
L_XY=L_YX.
\end{gathered}
\end{equation}
Add the compatibility equations with derivative directions $X$ and
$Y$, subtract the one with derivative direction $Z$, and use the
second identity in \eqref{eq:LC-linear-system}. The result is the
variation of the Koszul formula:
\begin{equation}\label{eq:LC-coordinate}
2g(L_XY,Z)=(\nabla_Xh)(Y,Z)+(\nabla_Yh)(X,Z)-(\nabla_Zh)(X,Y).
\end{equation}
This recovers the coordinate formula in
\cite[(4.37)]{DLS}; $L_X$ is not in general skew.

To identify the varying metrics by orthonormal coframes,
let the infinitesimal coframe endomorphism be
$U=\frac12h^\sharp+\chi$, with $\chi$ skew.
Differentiating the induced change of connection gives
$K^{LC,\chi}=L-\nabla U$. Consequently
\begin{equation}\label{eq:LC-spin}
\begin{split}
g(X,K^{LC,\chi}_T Y)
={}&\frac12\bigl((\nabla_Yh)(T,X)-(\nabla_Xh)(T,Y)\bigr)\\
&-g(X,(\nabla_T\chi)Y).
\end{split}
\end{equation}
The expression is skew and depends only on the metric
variation and the chosen frame identification.

Write a one-cochain as $m=s+q$,
with $s\in S^2T^*Y$, $q\in\Lambda^2T^*Y$ and $h=2s$.
Here the associated two-tensor is $m(V,X)=m_X(V)$, and $q$ is
its antisymmetric part.
For $q=0$, \eqref{eq:LC-spin} in the symmetric
frame $\chi=0$ is $\cK^1m$. More generally, identify
$q$ with the ordinary two-form and take $\chi=q^\sharp$.
Since
\[
\dd q(T,X,Y)=(\nabla_Tq)(X,Y)
             +(\nabla_Xq)(Y,T)+(\nabla_Yq)(T,X),
\]
substitution in \eqref{eq:LC-spin} proves
\begin{equation}\label{eq:LC-q}
\cK^1m=K^{LC,q}+\dd q.
\end{equation}
Here $K^{LC,q}$ abbreviates $K^{LC,q^\sharp}$, and a three-form
is viewed as a skew-endomorphism-valued one-form by
$g(X,(\dd q)_T^\sharp Y)=\dd q(T,X,Y)$.
Thus the cochain operator includes an explicit
antisymmetric-field contribution.

\begin{proposition}\label{prop:family-instanton}
Suppose $\varphi(t)$ is a torsion-free family and its
infinitesimal geometric tensor is written as
\begin{equation}\label{eq:form-variation}
\dot\varphi=\sum_um_u\wedge\iota_u\varphi,\qquad
\dot\psi=\sum_um_u\wedge\iota_u\psi.
\end{equation}
In a fixed identification of the tangent bundle, the corresponding variation $K^{LC}$ of the Levi--Civita connection satisfies
\begin{equation}\label{eq:linearized-instanton}
\pi_7\bigl(\dE K^{LC}+\ff m\bigr)=0.
\end{equation}
If, in addition, $d_0^1m=0$ and the chosen connection variation equals $\cK^1m$, for example for a symmetric closed variation in the symmetric frame, then \eqref{eq:linearized-instanton} is precisely the degree-one closed-cochain relation in \eqref{eq:first-primitive}.
\end{proposition}
\begin{proof}
Differentiating $R=\dd\Theta+\Theta\wedge\Theta$ gives
\[
\dot R=\dd K^{LC}+\Theta\wedge K^{LC}
                       +K^{LC}\wedge\Theta=\dE K^{LC}.
\]
The derivative of $R(t)\wedge\psi(t)=0$ is therefore
\begin{equation}\label{eq:instanton-family-derivative}
(\dE K^{LC})\wedge\psi+R\wedge\dot\psi=0.
\end{equation}
To express the second term, contract the background
instanton equation:
\[
0=\iota_u(R\wedge\psi)
 =(\iota_uR)\wedge\psi+R\wedge\iota_u\psi.
\]
Using \eqref{eq:form-variation} and keeping the exterior
order gives
\begin{align}
R\wedge\dot\psi
&=\sum_um_u\wedge R\wedge\iota_u\psi
 =-\sum_um_u\wedge(\iota_uR)\wedge\psi
\notag\\
&=\sum_u(\iota_uR)\wedge m_u\wedge\psi
 =(\ff m)\wedge\psi.
\label{eq:variation-Atiyah}
\end{align}
Wedge with $\psi$ kills precisely $\Lambda^2_{14}$
and is injective on $\Lambda^2_7$, by
the description of $\Lambda^2_{14}$ and the decomposition
\eqref{eq:decompositions}. Equations
\eqref{eq:instanton-family-derivative} and
\eqref{eq:variation-Atiyah} imply
\eqref{eq:linearized-instanton}.
\end{proof}

A change of tangent-frame identification adds
$\dE\lambda$ to $K^{LC}$; the extra term in
\eqref{eq:instanton-family-derivative} is
$[R,\lambda]\wedge\psi=0$. Thus \eqref{eq:linearized-instanton} is independent of the tangent-frame identification. Proposition~\ref{prop:K} itself is a cochain identity and does not require $m$ to integrate to a torsion-free family. Formula \eqref{eq:LC-q} keeps the antisymmetric cochain components distinct from ordinary metric variations.

\section{The codifferential homotopy for the order-\texorpdfstring{$\alpha'$}{alpha-prime} coupling}
\label{sec:second}

\subsection{Definition and domain}

For $a\in\Omega^p(Y,E)$, let
$a^-\in\Omega^p(Y,\Lambda^2T^*Y)$ be its metric-skew part:
\begin{equation}\label{eq:skew-a}
a^-(X,Y)=\frac12\bigl(g(X,aY)-g(Y,aX)\bigr).
\end{equation}
Regard $a^-$ as a $\Lambda^pT^*Y$-valued two-form through the
canonical exchange of tensor factors
\[
\Omega^p(Y,\Lambda^2T^*Y)
=\Gamma(\Lambda^pT^*Y\otimes\Lambda^2T^*Y)
\cong\Omega^2(Y,\Lambda^pT^*Y).
\]
In $d_\nabla^*a^-$, the covariant codifferential
\[
d_\nabla^*:\Omega^2(Y,\Lambda^pT^*Y)
\longrightarrow\Omega^1(Y,\Lambda^pT^*Y)
\]
uses the induced Levi--Civita connection on $\Lambda^pT^*Y$.
Its output is identified with $\Omega^p(Y,T^*Y)$ by exchanging
the two tensor factors again. Define $\ell^k:E^k\to G^k$ and its
weighted version by
\begin{equation}\label{eq:L-definition}
\ell a=\frac12d_\nabla^*a^-,
\qquad
L_E^k=\eps\ell^k:E^k\longrightarrow G_1^k.
\end{equation}
The same definition applies on unrestricted exterior forms.
The induced Levi--Civita connection preserves the skew and symmetric endomorphism summands. Both $b$ and $L_E$ depend only on the skew part $a^-$, while the coefficient bundle remains $E=\End_0(TY)$.
Indeed,
\[
(\bb_0a)(Z)=-\frac14\tau((\iota_ZR)\wedge a)
\]
depends only on $a^-$ because $R(Z,\cdot)$ is skew and the trace
pairing pairs skew and symmetric endomorphisms orthogonally.

Since $\nabla\Pi_k=0$, the covariant codifferential in
\eqref{eq:L-definition} commutes with the canonical projections:
\begin{equation}\label{eq:L-projection}
\Pi_k\ell=\ell\Pi_k.
\end{equation}

\begin{curvaturelemma}
If $S\in\mathfrak{so}(TY)$ corresponds to
$A(U,V)=g(U,SV)$, then
\begin{equation}\label{eq:L-skew-contraction}
\sum_j A(e_j,R(X,e_j)Z)
=-\frac12\tr\bigl(R(Z,X)S\bigr).
\end{equation}
\end{curvaturelemma}
\begin{proof}
Put $C(U)=R(X,U)Z$. Skewness of $S$, pair symmetry and cyclicity
of the trace give
\[
\begin{split}
\sum_j A(e_j,R(e_j,Z)X)
&=\sum_j g(e_j,R(X,Se_j)Z)\\
&=\tr(CS)=\tr(SC)
=\sum_j A(e_j,R(X,e_j)Z).
\end{split}
\]
Contract the first Bianchi identity
\begin{equation}\label{eq:L-Bianchi}
R(Z,X)e_j+R(X,e_j)Z+R(e_j,Z)X=0
\end{equation}
with $A(e_j,\cdot)$ and sum over $j$. The first term is
$\tr(SR(Z,X))$ and the other two are equal, proving the claim.
The identity applies coefficientwise to form-valued $S$.
\end{proof}

\begin{proposition}\label{prop:L}
The coupling $b$ is null-homotopic:
\begin{equation}\label{eq:L-homotopy}
b^k=d_1^kL_E^k-L_E^{k+1}e^k,\qquad k=0,1,2,3.
\end{equation}
\end{proposition}

\begin{proof}
As for $\cK$, we first work on unrestricted forms. Choose an
orthonormal frame that is normal at the point of calculation, and
extend $Z$ normally there.
Expanding the two derivative compositions then gives
\begin{equation}\label{eq:L-second-jets}
((\dn\ell-\ell\dE)a)(Z)
=-\frac12\sum_{i,j}e^i\wedge
\bigl([\nabla_{e_i},\nabla_{e_j}]a^-\bigr)(e_j,Z).
\end{equation}
The symmetric second-derivative terms cancel exactly, before any
curvature identities are used. The curvature of the induced
connection on $\Lambda^pT^*Y\otimes\Lambda^2T^*Y$ is the sum of
its actions on the form and $\Lambda^2T^*Y$ factors:
\begin{equation}\label{eq:L-commutator}
\begin{split}
\bigl([\nabla_{e_i},\nabla_{e_j}]a^-\bigr)(e_j,Z)
={}&-a^-(R(e_i,e_j)e_j,Z)-a^-(e_j,R(e_i,e_j)Z)\\
&+\rho_p(R(e_i,e_j))\bigl(a^-(e_j,Z)\bigr).
\end{split}
\end{equation}
The contraction over $j$ retains the first curvature term.

For the argument $Z$, the algebraic curvature lemma
gives
\begin{equation}\label{eq:L-b-term}
\frac12\sum_{i,j}e^i\wedge a^-(e_j,R(e_i,e_j)Z)
=-\frac14\tau((\iota_ZR)\wedge a)=(\bb_0a)(Z).
\end{equation}
For the remaining contraction,
$\sum_jR(e_i,e_j)e_j=\Ric^\sharp(e_i)$ with the convention
\eqref{eq:curvature-conventions}. Its contribution vanishes by
\eqref{eq:background-curvature}.

The form-factor term is reduced by
\eqref{eq:form-curvature-reduction}. Consequently the full
unprojected remainder is
\begin{equation}\label{eq:L-unprojected}
\begin{gathered}
(\dn\ell-\ell\dE)a=\bb_0a+\mathscr Q_\ell(a),\\
\mathscr Q_\ell(a)(Z)
=\frac12\sum_i e^i\wedge a^-(\Ric^\sharp(e_i),Z)
 +\frac12\sum_j\mathcal I_{R_{e_j}}\bigl(a^-(e_j,Z)\bigr).
\end{gathered}
\end{equation}
The first term is zero by Ricci-flatness. Both sums are metric
contractions and hence independent of the normal frame.

For a canonical input, \eqref{eq:L-projection}
proves
\begin{equation}\label{eq:L-projected-composition}
d_1^kL_E^ka-L_E^{k+1}e^ka
=\eps\Pi_{k+1}(\dn\ell-\ell\dE)a.
\end{equation}
In degree $0$, insertion into a zero-form is zero. In degree $1$
the inserted form is a scalar, so the remainder is a sum of
curvature forms in $\Lambda^2_{14}$ and has zero $\pi_7$
projection. In degree $2$ the inserted form is a one-form,
so each term lies in the span of $\eta_{14}\wedge\lambda$;
its scalar three-form projection vanishes by
Lemma~\ref{lem:projection}.
In degree $3$ the target is zero.
Applying these statements to
\eqref{eq:L-unprojected} and multiplying by $\eps$
gives \eqref{eq:L-homotopy} in all degrees.
\end{proof}

\subsection{Closed cochains and gauge transformations}

If $e^ka=0$, Proposition~\ref{prop:L} supplies the
explicit geometric primitive
\begin{equation}\label{eq:beta-primitive}
b^ka=d_1^k(L_E^ka).
\end{equation}
It applies to every closed gauge cochain. For $\lambda\in E^0$, the same identity
at degree zero says
\begin{equation}\label{eq:L-gauge-exact}
L_E^1e^0\lambda=d_1^0L_E^0\lambda-b^0\lambda.
\end{equation}
Applying $d_1^1$ and using $d_1^1d_1^0=0$
gives $b^1e^0\lambda=-d_1^1b^0\lambda$.
Thus exact bundle representatives have the required exact image.

Since $\ell$ is defined tensorially from $g$ and $\nabla$, the homotopy is globally defined. Section~\ref{subsec:transgression} relates this homotopy to anomaly transgression.

\section{Cancellation of the residual geometric coupling}
\label{sec:residual}

\subsection{The two curvature terms}

The $\cK$-transformation of \S\ref{subsec:K-splitting} replaces $a$ by $a'+\cK m$. In the geometric row it therefore replaces
$c$ by
\begin{equation}\label{eq:gamma-definition}
\gamma^k=c^k+b^k\cK^k.
\end{equation}

\begin{proposition}\label{prop:gamma}
The residual coupling vanishes:
\begin{equation}\label{eq:gamma-zero}
c^k+b^k\cK^k=0
\end{equation}
in every canonical degree. In fact
$\cc+\bb\cK=0$ before projection on arbitrary
exterior forms.
\end{proposition}
\begin{proof}
Recall that $g(X,\cB(m)Y)=(\nabla_Xm)_Y$. By
\eqref{eq:K-definition}, the metric identification gives
\[
(\cK m)^\sharp=\cB(m)^*-\cB(m)=-2\cB(m)^-,
\qquad \cB(m)^-=\tfrac12(\cB(m)-\cB(m)^*).
\]
Here the adjoint and skew part act on the endomorphism factor.
Since $R(Z,\cdot)$ is skew, the trace pairing satisfies
$\tau((\iota_ZR)\wedge \cB(m))=\tau((\iota_ZR)\wedge \cB(m)^-)$.
Using the trace convention \eqref{eq:tau},
\begin{align}
(\bb_0\cK m)(Z)
&=-\frac14\tau\bigl((\iota_ZR)\wedge(\cK m)^\sharp\bigr)
\notag\\
&=+\frac12\tau\bigl((\iota_ZR)\wedge \cB(m)\bigr).
\label{eq:gamma-bK}
\end{align}
On the other hand,
\begin{equation}\label{eq:gamma-c}
(\cc_0m)(Z)=-\frac12\tau\bigl((\iota_ZR)\wedge \cB(m)\bigr).
\end{equation}
Equations \eqref{eq:gamma-bK} and \eqref{eq:gamma-c}
cancel before projection. Multiplication by $\eps$
and application of the target projector prove
\eqref{eq:gamma-zero}.
\end{proof}

\subsection{Consistency of the complete differential}

\begin{proposition}\label{prop:nilpotency}
The differential \eqref{eq:D-definition} squares to zero
on the weighted coefficient complex. Both
\eqref{eq:first-SES} and \eqref{eq:second-SES}
are short exact sequences of complexes.
\end{proposition}
\begin{proof}
The diagonal squares vanish by
\eqref{eq:canonical-square-one}. Suppressing adjacent
degree labels, Proposition~\ref{prop:K} gives
\[
ef+fd_0
=e(\cK d_0-e\cK)+(\cK d_0-e\cK)d_0=0.
\]
Proposition~\ref{prop:L} similarly gives
\[
d_1b+be
=d_1(d_1L_E-L_Ee)+(d_1L_E-L_Ee)e=0.
\]
Finally, using $c=-b\cK$,
\begin{align}
d_1c+cd_0+bf
&=-d_1b\cK-b\cK d_0+b(\cK d_0-e\cK)
\notag\\
&=-(d_1b+be)\cK=0.
\label{eq:mixed-square-proof}
\end{align}
These are precisely the three off-diagonal blocks
of $D^{k+1}D^k$. The inclusion and quotient maps
already described in Section~\ref{subsec:extensions}
therefore give sequences of complexes.
\end{proof}

\subsection{Anomaly transgression}
\label{subsec:transgression}

The cancellation also has a geometric interpretation.
For a connection $A$ in a local trivialization,
define
\[
\CS_\tau(A)=\tau\bigl(A\wedge\dd A
                      +\tfrac23A\wedge A\wedge A\bigr).
\]
With $a=\delta A$, graded cyclicity and the exterior
Leibniz rule give
\begin{align}
\delta\CS_\tau(A)
&=\tau(a\wedge\dd A+A\wedge\dd a
                         +2a\wedge A\wedge A)
\notag\\
&=2\tau(a\wedge F)-\dd\tau(A\wedge a),
\label{eq:CS-variation}\\
\delta\tau(F\wedge F)
&=2\tau(\dd_Aa\wedge F)
 =2\dd\tau(a\wedge F).
\label{eq:pontryagin-variation}
\end{align}
In the first calculation,
$\dd\tau(A\wedge a)=\tau(\dd A\wedge a-A\wedge\dd a)$.
In the second, invariance of $\tau$ replaces ordinary
by covariant differentiation inside the trace, and
$\dd_AF=0$ removes the remaining term.

At $F=R$, the covariant transgression
$\mathcal T(a)=2\tau(a\wedge R)$ is global,
although $\tau(A\wedge a)$ in
\eqref{eq:CS-variation} is generally only local.
For a one-cochain, contraction gives
\begin{equation}\label{eq:T-contraction}
\iota_r\mathcal T(a)
=2\tau(a_rR+\cR_r\wedge a).
\end{equation}
The first term has exterior type 14, so
\begin{equation}\label{eq:beta-transgression}
(b^1a)_r=-\frac{\eps}{8}\,
                   \pi_7\iota_r\mathcal T(a).
\end{equation}
The same sign and factor are obtained directly from
the curvature pairing in \eqref{eq:b-unprojected}.
To connect the transgression to the global primitive,
apply the algebraic curvature lemma
\eqref{eq:L-skew-contraction}:
\begin{align}
\pi_7\iota_Z\mathcal T(a)
&=2\pi_7\tau\bigl((\iota_ZR)\wedge a\bigr)
\notag\\
&=-4\pi_7\sum_{i,j}e^i\wedge a^-(e_j,R(e_i,e_j)Z)
\notag\\
&=4\pi_7\sum_{i,j}e^i\wedge
   \bigl([\nabla_{e_i},\nabla_{e_j}]a^-\bigr)(e_j,Z).
\label{eq:transgression-commutator}
\end{align}
The last line uses \eqref{eq:L-commutator};
the other terms vanish by Ricci-flatness and the
degree-one curvature projection already proved.
Together with \eqref{eq:L-second-jets},
this recovers \eqref{eq:L-homotopy} in degree one.
This gives exactness in the canonical covector-valued complex.

For the gauge and tangent connection variations
$a$ and $k_{\rm tan}$ at the common background,
\eqref{eq:pontryagin-variation} gives
\begin{equation}\label{eq:anomaly-transgression}
\delta\{\tau(F\wedge F)-\tau(R\wedge R)\}
=2\dd\tau\bigl((a-k_{\rm tan})\wedge R\bigr).
\end{equation}
Here $k_{\rm tan}$ denotes the induced tangent-connection variation. If the gauge and tangent-connection variations coincide, $a=k_{\rm tan}=\cK m$, the two covariant
transgressions agree pointwise; the local boundary
terms also cancel in trivializations related by the fixed standard-embedding identification $V\simeq TY$.
The gauge insertion is
$-\eps\tau(\cR_r\wedge\cK m)/4$, whereas the
gravitational insertion is its negative.
The latter is exactly $c m$ by
\eqref{eq:gamma-bK}--\eqref{eq:gamma-c}.
This explains \eqref{eq:gamma-zero} as the
differential expression of anomaly cancellation.

Equation \eqref{eq:LC-q} relates $\cK^1m$ to the Levi--Civita variation, including the antisymmetric-field contribution. For arbitrary one-cochains, the algebraic identity inserts the same formal variation $\cK m$ in the gauge and tangent-connection sectors. After
$a=a'+\cK m$, \eqref{eq:anomaly-transgression}
retains the independent source from $a'$.
Its canonical image is separately exact by
Proposition~\ref{prop:L}. Thus the two sectors are removed independently.

\Needspace{25\baselineskip}
\section{The decoupling theorem}
\label{sec:theorem}

\begin{theorem}\label{thm:main}
Let $(Y^7,\varphi)$ be a torsion-free $G_2$ manifold.
At the minimal standard embedding
\eqref{eq:standard-embedding}, take
$E=\End_0(TY)$, the invariant form $\tau$ of \eqref{eq:tau},
and the weighted coefficient complex
\eqref{eq:weighted-module} with differential
\eqref{eq:D-definition}. For every canonical
degree $k=0,1,2,3$, the maps
\begin{equation}\label{eq:VW}
V^k=
\begin{pmatrix}
1&-L_E^k&0\\
0&1&\cK^k\\
0&0&1
\end{pmatrix},
\qquad
W^k=
\begin{pmatrix}
1&L_E^k&-L_E^k\cK^k\\
0&1&-\cK^k\\
0&0&1
\end{pmatrix}
\end{equation}
are mutually inverse differential operators on
smooth sections and satisfy
\begin{equation}\label{eq:main-conjugation}
W^{k+1}D^kV^k=\diag(d_1^k,e^k,d_0^k).
\end{equation}
In particular,
\begin{equation}\label{eq:main-splitting}
\boxed{\;
(B_\tau^\bullet,D)\cong
(G_1^\bullet,d_1)\oplus(E^\bullet,e)
                    \oplus(G_0^\bullet,d_0)
\;}
\end{equation}
as complexes of real smooth sections. The maps
also preserve the dual-number action, and give
an $S$-linear chain isomorphism
\begin{equation}\label{eq:S-splitting}
(B_\tau^\bullet,D)\cong
(G^\bullet\otimes_\R S,d\otimes1)
                       \oplus(E^\bullet,e).
\end{equation}
No compactness or full-holonomy hypothesis is required.
\end{theorem}

\begin{proof}
Extend the first transformation by the identity on $G_1$:
\[
T_K^k=
\begin{pmatrix}1&0&0\\0&1&\cK^k\\0&0&1\end{pmatrix},
\qquad
S_K^k=
\begin{pmatrix}1&0&0\\0&1&-\cK^k\\0&0&1\end{pmatrix}.
\]
Direct block multiplication gives
\begin{equation}\label{eq:first-full-conjugation}
S_K^{k+1}D^kT_K^k=
\begin{pmatrix}
d_1^k&b^k&c^k+b^k\cK^k\\
0&e^k&f^k+e^k\cK^k-\cK^{k+1}d_0^k\\
0&0&d_0^k
\end{pmatrix}.
\end{equation}
Propositions~\ref{prop:K} and~\ref{prop:gamma}
set the two entries in the last column above
$d_0$ to zero. The remaining transformation is
\[
U^k=
\begin{pmatrix}1&-L_E^k&0\\0&1&0\\0&0&1\end{pmatrix},
\qquad
(U^k)^{-1}=
\begin{pmatrix}1&L_E^k&0\\0&1&0\\0&0&1\end{pmatrix}.
\]
Its top-middle transformed entry is
$b^k-d_1^kL_E^k+L_E^{k+1}e^k$, which vanishes
by Proposition~\ref{prop:L}. The products
$V=T_KU$ and $W=U^{-1}S_K$ are precisely
\eqref{eq:VW}.

At a fixed degree, direct multiplication gives
\begin{align}
WV&=
\begin{pmatrix}
1&-L_E+L_E&L_E\cK-L_E\cK\\
0&1&\cK-\cK\\
0&0&1
\end{pmatrix}=\id,
\label{eq:WV-product}\\
VW&=
\begin{pmatrix}
1&L_E-L_E&-L_E\cK+L_E\cK\\
0&1&-\cK+\cK\\
0&0&1
\end{pmatrix}=\id.
\label{eq:VW-product}
\end{align}
The top-right term $-L_E\cK$ in $W$ cancels the composition
created by the two off-diagonal entries.

With the degree labels retained, one computes
\begin{align}
D^kV^k
&=
\begin{pmatrix}
d_1^k&b^k-d_1^kL_E^k&c^k+b^k\cK^k\\
0&e^k&f^k+e^k\cK^k\\
0&0&d_0^k
\end{pmatrix}
\notag\\
&=
\begin{pmatrix}
d_1^k&-L_E^{k+1}e^k&0\\
0&e^k&\cK^{k+1}d_0^k\\
0&0&d_0^k
\end{pmatrix}
=V^{k+1}\diag(d_1^k,e^k,d_0^k).
\label{eq:DV-calculation}
\end{align}
Multiplication on the left by $W^{k+1}$,
using \eqref{eq:WV-product}, proves
\eqref{eq:main-conjugation}. This includes
degree $3$, where all outgoing arrows are zero.

Finally, in the unscaled coefficient coordinates
of \eqref{eq:N-action}, the transformations are
\begin{align}
V(x',a',m)&=(x'-\ell a',\,a'+\cK m,\,m),
\label{eq:V-coefficients}\\
W(x,a,m)&=(x+\ell a-\ell\cK m,\,a-\cK m,\,m).
\label{eq:W-coefficients}
\end{align}
Applying $N$ before or after either map gives
$(m,0,0)$. Hence $VN=NV$ and $WN=NW$,
which proves $S$-linearity with the specified
action and establishes \eqref{eq:S-splitting}.
\end{proof}

\subsection{Connecting maps and cohomology}

The connecting maps can be computed directly at chain level. For
\eqref{eq:first-SES}, lift a closed $m$ to
$(0,m)$ and apply the differential. The result
is $(fm,0)$, so
\begin{equation}\label{eq:delta-one}
\delta_1^k[m]=[f^km]=[-e^k\cK^km]=0.
\end{equation}
For the second sequence, a closed element
$(a,m)$ of $A^k$ satisfies
$d_0m=0$ and $ea+fm=0$. Lifting it to
$(0,a,m)$ yields
\begin{equation}\label{eq:delta-two}
\delta_2^k[(a,m)]=[b^ka+c^km].
\end{equation}
Put $a'=a-\cK m$. The first splitting gives
$ea'=0$, and Propositions~\ref{prop:L}
and~\ref{prop:gamma} give
\begin{equation}\label{eq:delta-two-primitive}
b^ka+c^km=b^k(a-\cK^km)
          =d_1^kL_E^k(a-\cK^km).
\end{equation}
This is the explicit primitive for the second
connecting class. Equivalently, in split
coordinates $\delta_2([a',m])=[ba']$,
while the geometric contribution is zero
already as a cochain map.

The closed lift in the second
extension is
\begin{equation}\label{eq:second-chain-section}
j^k(a,m)=\bigl(-L_E^k(a-\cK^km),\,a,\,m\bigr).
\end{equation}
Let $q$ and $\iota_2$ denote the quotient and inclusion in
\eqref{eq:second-SES}, respectively. Then $qj=\id$. Moreover,
$S_KD_A=\diag(e,d_0)S_K$ and
$b=d_1L_E-L_Ee$ show that its top component
intertwines the differentials; the lower
components do so by definition.
A complementary chain retraction is
\begin{equation}\label{eq:second-chain-retraction}
r^k(u,a,m)=u+L_E^k(a-\cK^km).
\end{equation}
Its compositions satisfy $r\iota_2=\id$,
$rj=0$, and $\iota_2r+jq=\id$.
These are real chain maps; their relationship
to the $S$-module extension is explained in
Appendix~\ref{app:coefficients}.

\begin{corollary}\label{cor:cohomology}
In every degree,
\begin{equation}\label{eq:H-decomposition}
H^k(B_\tau,D)\cong
H^k(G_1,d_1)\oplus H^k(E,e)\oplus H^k(G_0,d_0).
\end{equation}
With the coefficient action retained, this is
equivalently
\begin{equation}\label{eq:H-S-decomposition}
H^k(B_\tau,D)\cong
\bigl(H^k(G,d)\otimes_\R S\bigr)\oplus H^k(E,e).
\end{equation}
\end{corollary}
\begin{proof}
Since $W^{k+1}D^k=D_{\rm diag}^kW^k$,
$W$ maps closed cochains to closed cochains
and exact cochains to exact cochains. The
same is true of $V$, and their induced maps
remain inverse. This proves
\eqref{eq:H-decomposition}. In coefficient
coordinates $d\otimes1$ acts separately on
the constant and $\eps$ terms, so its kernel
and image are the corresponding two
coefficients of $\ker d$ and $\im d$.
The quotient is $H^k(G)\otimes S$, with
$N$ sending the constant class to the
$\eps$ class. Together with
\eqref{eq:V-coefficients}--\eqref{eq:W-coefficients},
this proves \eqref{eq:H-S-decomposition}.
\end{proof}

Both long exact sequences consequently break
into short exact sequences
$0\to H^k(E)\to H^k(A)\to H^k(G_0)\to0$
and
$0\to H^k(G_1)\to H^k(B_\tau)\to H^k(A)\to0$
of real vector spaces, with the displayed
chain-induced sections in every degree.

\section{Relation to the Calabi--Yau standard embedding}
\label{sec:CY}

The standard-embedding proof in \cite[Section~3]{CMPSS}\footnote{Equation numbers refer to the published version. In arXiv:2409.04350v2, the equations following the second extension are numbered one lower.}
has two structural features also visible here. A
connection variation removes the first curvature
coupling, and substituting that variation into the
second extension cancels an explicit geometric
curvature term. The remaining bundle source must
then be shown to be exact. The geometry used for
these steps is different in the two settings.

In the Calabi--Yau first extension, the local
connection variation is a holomorphic covariant
derivative of the complex-structure cochain:
\begin{equation}\label{eq:CY-response}
a^\sigma{}_\tau
=\nabla_\tau\Delta^\sigma+(a_0)^\sigma{}_\tau
\end{equation}
in the convention of \cite[(3.8)]{CMPSS}.
Commuting the antiholomorphic differential with
that derivative, and using the K\"ahler curvature
symmetries, produces the Atiyah curvature
insertion for a closed $\Delta$.
The real skew derivative $\cK$ plays the
corresponding role here. Its curvature
commutator contains the additional form-factor curvature
terms in \eqref{eq:K-unprojected}; their
disappearance requires both the exterior
instanton condition and the two separate
canonical projections.

After substituting \eqref{eq:CY-response}
into their second extension,
\cite[(3.16)--(3.17)]{CMPSS} cancels the
geometric derivative-curvature contribution
and retains the pairing of curvature with
$a_0$. The corresponding cancellation in
our convention is
$c+b\cK=0$. Its unprojected proof
\eqref{eq:gamma-bK}--\eqref{eq:gamma-c}
uses the real matrix trace and the skew
tangent-connection variation. In both cases the paired
gauge and gravitational transgressions
explain why the induced variation is the
relevant one. The second extension also contains the independent bundle deformation $a_0$.

For that remaining source, the argument in
\cite[(3.18)--(3.19)]{CMPSS} uses harmonic
test forms, the holomorphic volume form,
K\"ahler identities and integration by
parts. Antisymmetrization
over four holomorphic indices in complex
dimension three is part of the calculation.
In the $G_2$ proof, the explicit covariant-codifferential homotopy $\ell$ replaces this Hodge-theoretic argument. Ricci-flatness removes its
contracted endomorphism-index curvature
term, and
\[
\pi_7 R_{14}=0,\qquad
\pi_1(R_{14}\wedge\lambda_1)=0
\]
remove the form-factor curvature terms.
These are precisely the dimension-seven
projection identities proved in
Section~\ref{sec:geometry}.

Once a null-homotopy is known, the passage
to a triangular chain transformation is
homological algebra.
The geometric content is the construction
and verification of that homotopy.
The present proof gives an all-degree chain isomorphism by differential operators for the specified coefficient complex. This is a more explicit
chain-level conclusion than the
deformation-cohomology decomposition
displayed in \cite[(3.20)]{CMPSS}. The two constructions use different coefficient bundles and perturbative field identifications.

\section{Physical interpretation and scope}
\label{sec:discussion}

Theorem~\ref{thm:main} removes all internal extension couplings
from the physical differential after elimination of the independent
tangent-connection variation. It applies to the differential
and weighted coefficient module of Section~\ref{sec:heterotic}
on any torsion-free $G_2$ manifold, including compact backgrounds
with full or reduced holonomy.

In $M=m+\eps x$, the summands $G_0$ and $G_1$ are the zeroth-
and first-order coefficients of the same geometric field.
Equation~\eqref{eq:H-S-decomposition} expresses the cohomology
decomposition with the coefficient action preserved.

The gauge sector is $E=\End_0(TY)$ at $V=TY$. A complete
$E_8\times E_8$ or $SO(32)$ spectrum additionally includes the
centralizer and other representation components, whose bundles
and physical multiplicities lie beyond the present theorem.

Theorem~\ref{thm:main} concerns the linear coefficient complex at the
standard embedding. The corresponding nonlinear deformation problem has
a different geometric organization, which we discuss next.

\section{Outlook}
\label{sec:outlook}

The chain isomorphism of Theorem~\ref{thm:main} concerns the coefficient
complex of Section~3. A corresponding nonlinear gauge--geometry problem can
instead be formulated directly in terms of the geometric and gauge fields.
Finite-deformation $L_\infty$ methods in the complex-geometric heterotic
setting were developed in \cite{FiniteDeformations}.
At $\alpha'=0$, let $\varphi$ be a torsion-free $G_2$ structure. Using the
fixed standard-embedding identification $V\simeq TY$, regard $A$ as a
connection on $TY$ and restrict attention to connections for which
$q=A-\nabla^\varphi\in\Omega^1(Y,E)$. The relative connection
\begin{equation}\label{eq:outlook-relative}
q=A-\nabla^\varphi\in\Omega^1(Y,E)
\end{equation}
measures the deviation from the standard embedding. Its curvature satisfies
\begin{equation}\label{eq:outlook-curvature}
F_A=R_\varphi+\dd_{\nabla^\varphi}q+\frac12[q,q].
\end{equation}
Here $[q,q]=2q\wedge q$ is the graded commutator.
Since $R_\varphi\wedge\psi_\varphi=0$ for torsion-free $\varphi$, the
instanton equation becomes
\begin{equation}\label{eq:outlook-instanton}
\left(\dd_{\nabla^\varphi}q+\frac12[q,q]\right)\wedge\psi_\varphi=0.
\end{equation}
Thus $q=0$ is the nonlinear standard-embedding locus.

This formulation is triangular at $\alpha'=0$: the torsion-free $G_2$
equations determine the geometry independently of $q$, whereas the
instanton equation depends on $\varphi$ through both $\nabla^\varphi$
and the $G_2$ projection. Accordingly, even though the unary deformation
problem is split in relative-connection variables, the instanton
deformation complex varies with the $G_2$ structure. Let $\varphi_s$ be a smooth torsion-free family. After choosing a smooth identification of the graded bundles along the family, let $\mathscr D_s$ denote the corresponding instanton differential on a fixed graded space. Then
\begin{equation}\label{eq:outlook-rho}
\rho_u=\left.\frac{d}{ds}\right|_{s=0}\mathscr D_s,
\qquad u=\dot\varphi_0,
\end{equation}
satisfies
\begin{equation}\label{eq:outlook-chain}
\mathscr D_0\rho_u+\rho_u\mathscr D_0=0.
\end{equation}
It therefore induces a map on instanton cohomology,
\begin{equation}\label{eq:outlook-cohomology}
H^1_{\mathscr D_0}\longrightarrow H^2_{\mathscr D_0},\qquad [\eta]\longmapsto[\rho_u\eta],
\end{equation}
which measures the first-order obstruction to extending an instanton
deformation class along the geometric direction $u$. Write
\[
\eta_s=\eta+s\eta_1+O(s^2),\qquad
\mathscr D_s\eta_s=O(s^2).
\]
Expanding to first order gives
\[
\mathscr D_0\eta_1+\rho_u\eta=0.
\]
Thus the class $[\rho_u\eta]\in H^2_{\mathscr D_0}$ is the obstruction to
extending $[\eta]$ to first order along the family $\varphi_s$.

The nonlinear coupling is therefore not captured by transporting a
fixed $L_\infty$ structure on the coefficient complex through the chain
maps of Theorem~\ref{thm:main}. Rather, the instanton deformation algebra
itself varies over the deformation space of torsion-free $G_2$
structures. The variation of its differential gives the first mixed
interaction, while the variation of its projected bracket and the
corresponding higher coherences organize the subsequent nonlinear terms.
Thus the linear splitting may not extend to a direct-product nonlinear
deformation problem: already at quadratic order, variation of the
instanton complex along torsion-free $G_2$ deformations produces mixed
gauge--geometry operations. The resulting higher structure, together
with its first-order $\alpha'$ corrections, will be studied separately.

\section*{Acknowledgements}
Generative AI, specifically OpenAI's ChatGPT 6 Astra, was used to generate and validate tensor identities appearing in the proofs through symbolic computations in Cadabra and SageMath, and to assist in structuring the prose of this manuscript.

\appendix
\section{Projector normalizations and form-degree conversion}
\label{app:realization}

Choose a local oriented orthonormal coframe
$e^1,\ldots,e^7$, with dual frame $e_1,\ldots,e_7$. Our
normalization is
\begin{align}
\varphi
&=e^{123}+e^{145}+e^{167}+e^{246}
-e^{257}-e^{347}-e^{356},
\label{eq:phi}\\
\varphi\wedge\psi&=7\vol_g,\qquad
\varphi_{imn}\varphi_j{}^{mn}=6g_{ij}.
\label{eq:phi-normalization}
\end{align}
With respect to this normalization, the orthogonal projectors used in
Section~2 are
\begin{align}
(\pi^2_7\eta)_{ij}
&=\frac16\varphi_{ijk}\varphi^{abk}\eta_{ab},
&\pi^2_{14}&=1-\pi^2_7,
\label{eq:p2}\\
\pi^3_1\zeta
&=\frac{\zeta_{abc}\varphi^{abc}}{42}\varphi,
&\pi^3_7\zeta
&=\frac14\sum_i
\langle\iota_{e_i}\psi,\zeta\rangle\iota_{e_i}\psi,
\label{eq:p3}\\
\pi^3_{27}&=1-\pi^3_1-\pi^3_7.&&\notag
\end{align}
For the form-degree conversion, define
\[
J_0=J_1=\id,\qquad
J_2(\eta)=\eta\wedge\psi,\qquad
J_3(\zeta)=\zeta\wedge\psi.
\]
The normalization identities

$$
*(\iota_v\varphi\wedge\psi)=3v^\flat,
\qquad
(t\varphi)\wedge\psi=7t\vol_g
$$

give the inverse maps:
\begin{equation}\label{eq:J-inverses}
J_2^{-1}\omega
=\frac13\iota_{(*\omega)^\sharp}\varphi
\quad(\omega\in\Lambda^6),
\qquad
J_3^{-1}\nu
=\frac{*\nu}{7}\varphi
\quad(\nu\in\Lambda^7).
\end{equation}
Moreover,
\begin{equation}\label{eq:J-projections}
(\pi_7\eta)\wedge\psi=\eta\wedge\psi,
\qquad
(\pi_1\zeta)\wedge\psi=\zeta\wedge\psi,
\end{equation}
for every two-form $\eta$ and three-form $\zeta$. The first identity
uses $\Lambda^2_{14}=\ker(\,\cdot\,\wedge\psi)$, and the second follows
from orthogonal projection onto $\R\varphi$. Both identities apply
coefficientwise to bundle-valued forms.

Let

$$
\widehat d_W^k
 =J_{k+1}\check\dd_W^kJ_k^{-1}.
$$

The degree-zero arrow is $\dd_W$. In degree one,
\eqref{eq:J-projections} gives

$$
\widehat d_W^1a
 =(\pi_7\dd_Wa)\wedge\psi
 =\psi\wedge\dd_Wa.
$$

For a degree-two canonical representative $\eta$,

$$
J_3\check\dd_W^2\eta
 =(\dd_W\eta)\wedge\psi
 =\dd_W(\eta\wedge\psi),
$$

since $\dd\psi=0$. Thus $\widehat d_W^2=\dd_W$ on six-forms,
and the outgoing seven-form arrow is zero. This is the
$0,1,6,7$ realization of the canonical complex, with canonical
degree still equal to $0,1,2,3$.

\section{The coefficient module and its cohomology}
\label{app:coefficients}

The geometric dual-number module decomposes over $\R$ as

$$
G^\bullet\otimes_\R S
 =G^\bullet\oplus\eps G^\bullet.
$$

In the ordering $(x,a,m)$, multiplication by $\eps$ is
\begin{equation}\label{eq:N-matrix}
N=
\begin{pmatrix}
0&0&1\\
0&0&0\\
0&0&0
\end{pmatrix},
\qquad N^2=0.
\end{equation}
The unscaled coefficient differential is
\begin{equation}\label{eq:D-unscaled}
D_{\rm coeff}
=
\begin{pmatrix}
d&b_0&c_0\\
0&e&f\\
0&0&d
\end{pmatrix}.
\end{equation}
Both $D_{\rm coeff}N$ and $ND_{\rm coeff}$ have only the
top-right block $d$, so

$$
D_{\rm coeff}N=ND_{\rm coeff}.
$$

The diagonal differential has the same property, and
\eqref{eq:V-coefficients}--\eqref{eq:W-coefficients} show that the
chain isomorphism commutes with $N$. This proves preservation of the
$S$-module structure in \eqref{eq:S-splitting}.

The second exact sequence is a sequence of $S$-modules, but its
chain section \eqref{eq:second-chain-section} is only $\R$-linear.
This is already forced by
\begin{equation}\label{eq:dual-number-SES}
0\longrightarrow\eps S
\longrightarrow S
\longrightarrow S/(\eps)
\longrightarrow0.
\end{equation}
Indeed, an $S$-linear section sending $1$ to $1+\eps t$ would give

$$
0=s(\eps\cdot1)
 =\eps\,s(1)
 =\eps(1+\eps t)
 =\eps,
$$

a contradiction. The same obstruction applies to every nonzero
geometric cochain.

Finally,

$$
d(m+\eps x)=0
\quad\Longleftrightarrow\quad
dm=dx=0,
$$

and the exact elements are $dv+\eps dw$. Hence the diagonal
geometric cohomology is $H(G)\otimes_\R S$. Under the full chain
isomorphism, a closed coefficient cochain $(x,a,m)$ is represented by
\begin{equation}\label{eq:explicit-H-map}
\bigl(
[x+\ell(a-\cK m)],
[a-\cK m],
[m]
\bigr).
\end{equation}
The chain relation shows that changing $(x,a,m)$ by an exact
$D_{\rm coeff}$-cochain changes the three entries by diagonal
coboundaries. Thus \eqref{eq:explicit-H-map} is well-defined, with
inverse induced by \eqref{eq:V-coefficients}.

\bibliographystyle{amsplain}
\bibliography{references}

@article{Atiyah,
  author = {Atiyah, Michael F.},
  title = {Complex analytic connections in fibre bundles},
  journal = {Trans. Amer. Math. Soc.},
  volume = {85},
  number = {1},
  year = {1957},
  pages = {181--207},
  doi = {10.1090/S0002-9947-1957-0086359-5},
  note = {\href{https://doi.org/10.1090/S0002-9947-1957-0086359-5}{doi:10.1090/S0002-9947-1957-0086359-5}}
}

@article{CHSW,
  author = {Candelas, Philip and Horowitz, Gary T. and Strominger, Andrew and Witten, Edward},
  title = {Vacuum configurations for superstrings},
  journal = {Nuclear Phys. B},
  volume = {258},
  year = {1985},
  pages = {46--74},
  doi = {10.1016/0550-3213(85)90602-9},
  note = {\href{https://doi.org/10.1016/0550-3213(85)90602-9}{doi:10.1016/0550-3213(85)90602-9}}
}

@article{FiniteDeformations,
  author = {Ashmore, Anthony and {de la Ossa}, Xenia and Minasian, Ruben and Strickland-Constable, Charles and Svanes, Eirik Eik},
  title = {Finite deformations from a heterotic superpotential: holomorphic {Chern--Simons} and an {$L_\infty$} algebra},
  journal = {J. High Energy Phys.},
  volume = {10},
  year = {2018},
  pages = {179},
  doi = {10.1007/JHEP10(2018)179},
  eprint = {1806.08367},
  archivePrefix = {arXiv},
  note = {\href{https://doi.org/10.1007/JHEP10(2018)179}{doi:10.1007/JHEP10(2018)179}; \href{https://arxiv.org/abs/1806.08367}{arXiv:1806.08367}}
}

@article{FernandezGray,
  author = {Fern{\'a}ndez, Marisa and Gray, Alfred},
  title = {Riemannian manifolds with structure group {$G_2$}},
  journal = {Ann. Mat. Pura Appl.},
  volume = {132},
  year = {1982},
  pages = {19--45},
  doi = {10.1007/BF01760975},
  note = {\href{https://doi.org/10.1007/BF01760975}{doi:10.1007/BF01760975}}
}

@article{Carrion,
  author = {{Reyes Carri{\'o}n}, Ram{\'o}n},
  title = {A generalization of the notion of instanton},
  journal = {Differential Geom. Appl.},
  volume = {8},
  number = {1},
  year = {1998},
  pages = {1--20},
  doi = {10.1016/S0926-2245(97)00013-2},
  note = {\href{https://doi.org/10.1016/S0926-2245(97)00013-2}{doi:10.1016/S0926-2245(97)00013-2}}
}

@article{FernandezUgarte,
  author = {Fern{\'a}ndez, Marisa and Ugarte, Luis},
  title = {Dolbeault cohomology for {$G_2$}-manifolds},
  journal = {Geom. Dedicata},
  volume = {70},
  year = {1998},
  pages = {57--86},
  doi = {10.1023/A:1004940807017},
  note = {\href{https://doi.org/10.1023/A:1004940807017}{doi:10.1023/A:1004940807017}}
}

@article{DLS,
  author = {{de la Ossa}, Xenia and Larfors, Magdalena and Svanes, Eirik Eik},
  title = {Infinitesimal moduli of {$G_2$} holonomy manifolds with instanton bundles},
  journal = {J. High Energy Phys.},
  volume = {11},
  year = {2016},
  pages = {016},
  doi = {10.1007/JHEP11(2016)016},
  eprint = {1607.03473},
  archivePrefix = {arXiv},
  note = {\href{https://doi.org/10.1007/JHEP11(2016)016}{doi:10.1007/JHEP11(2016)016};
          \href{https://arxiv.org/abs/1607.03473v2}{arXiv:1607.03473v2}}
}

@article{DLSHeterotic,
  author = {{de la Ossa}, Xenia and Larfors, Magdalena and Svanes, Eirik Eik},
  title = {The infinitesimal moduli space of heterotic {$G_2$} systems},
  journal = {Commun. Math. Phys.},
  volume = {360},
  number = {2},
  year = {2018},
  pages = {727--775},
  doi = {10.1007/s00220-017-3013-8},
  eprint = {1704.08717},
  archivePrefix = {arXiv},
  note = {\href{https://doi.org/10.1007/s00220-017-3013-8}{doi:10.1007/s00220-017-3013-8};
          \href{https://arxiv.org/abs/1704.08717}{arXiv:1704.08717}}
}

@article{MSS,
  author = {McOrist, Jock and Sticka, Martin and Svanes, Eirik Eik},
  title = {The physical moduli of heterotic {$G_2$} string compactifications},
  journal = {J. High Energy Phys.},
  volume = {05},
  year = {2025},
  pages = {219},
  doi = {10.1007/JHEP05(2025)219},
  eprint = {2409.13080},
  archivePrefix = {arXiv},
  note = {\href{https://doi.org/10.1007/JHEP05(2025)219}{doi:10.1007/JHEP05(2025)219};
          \href{https://arxiv.org/abs/2409.13080v3}{arXiv:2409.13080v3}}
}

@article{MSSMetric,
  author = {McOrist, Jock and Sticka, Martin and Svanes, Eirik Eik},
  title = {The heterotic {$G_2$} moduli space metric},
  journal = {J. High Energy Phys.},
  volume = {11},
  year = {2025},
  pages = {016},
  doi = {10.1007/JHEP11(2025)016},
  eprint = {2502.16093},
  archivePrefix = {arXiv},
  note = {\href{https://doi.org/10.1007/JHEP11(2025)016}{doi:10.1007/JHEP11(2025)016};
          \href{https://arxiv.org/abs/2502.16093v1}{arXiv:2502.16093v1}}
}

@article{CMPSS,
  author = {Chisamanga, Beatrice and McOrist, Jock and Picard, Sebastien and Svanes, Eirik Eik},
  title = {The decoupling of moduli about the standard embedding},
  journal = {J. High Energy Phys.},
  volume = {01},
  year = {2025},
  pages = {032},
  doi = {10.1007/JHEP01(2025)032},
  eprint = {2409.04350},
  archivePrefix = {arXiv},
  note = {\href{https://doi.org/10.1007/JHEP01(2025)032}{doi:10.1007/JHEP01(2025)032};
          \href{https://arxiv.org/abs/2409.04350v2}{arXiv:2409.04350v2}}
}

@article{BPS,
  author = {Kupka, Julian and Strickland-Constable, Charles and Svanes, Eirik Eik and Tennyson, David and Valach, Fridrich},
  title = {{BPS} complexes and {Chern--Simons} theories from {$G$}-structures in gauge theory and gravity},
  journal = {J. High Energy Phys.},
  volume = {10},
  year = {2025},
  pages = {192},
  doi = {10.1007/JHEP10(2025)192},
  eprint = {2406.03550},
  archivePrefix = {arXiv},
  note = {\href{https://doi.org/10.1007/JHEP10(2025)192}{doi:10.1007/JHEP10(2025)192};
          \href{https://arxiv.org/abs/2406.03550v1}{arXiv:2406.03550v1}}
}

@incollection{QuantumG2,
  author = {{de la Ossa}, Xenia and Larfors, Magdalena and Magill, Matthew and Svanes, Eirik E.},
  title = {Quantum aspects of heterotic {$G_2$} systems},
  booktitle = {2024 MATRIX Annals, Part II},
  series = {MATRIX Book Series},
  volume = {8},
  publisher = {Springer},
  year = {2026},
  pages = {253--293},
  doi = {10.1007/978-3-032-16206-9_16},
  eprint = {2412.14715},
  archivePrefix = {arXiv},
  note = {\href{https://doi.org/10.1007/978-3-032-16206-9_16}{doi:10.1007/978-3-032-16206-9\_16};
          \href{https://arxiv.org/abs/2412.14715v4}{arXiv:2412.14715v4}}
}

@article{Coupled,
  author = {Garcia-Fernandez, Mario and Lotay, Jason D. and {S{\'a} Earp}, Henrique N. and {da Silva Jr.}, Agnaldo A.},
  title = {Coupled {$G_2$}-instantons},
  journal = {Internat. J. Math.},
  volume = {37},
  number = {05n06},
  year = {2026},
  pages = {2542002},
  doi = {10.1142/S0129167X25420029},
  eprint = {2404.12937},
  archivePrefix = {arXiv},
  note = {\href{https://doi.org/10.1142/S0129167X25420029}{doi:10.1142/S0129167X25420029};
          \href{https://arxiv.org/abs/2404.12937}{arXiv:2404.12937}}
}

@article{CGT,
  author = {Clarke, Andrew and Garcia-Fernandez, Mario and Tipler, Carl},
  title = {{$T$}-dual solutions and infinitesimal moduli of the {$G_2$}-{Strominger} system},
  journal = {Adv. Theor. Math. Phys.},
  volume = {26},
  number = {6},
  year = {2022},
  pages = {1669--1704},
  doi = {10.4310/ATMP.2022.v26.n6.a3},
  eprint = {2005.09977},
  archivePrefix = {arXiv},
  note = {\href{https://doi.org/10.4310/ATMP.2022.v26.n6.a3}{doi:10.4310/ATMP.2022.v26.n6.a3};
          \href{https://arxiv.org/abs/2005.09977}{arXiv:2005.09977}}
}
\end{document}